\documentclass[11pt]{amsart}

\usepackage{amsmath,amssymb}
\usepackage{mathtools}
\usepackage{amsthm}
\usepackage[abbrev]{amsrefs}
\usepackage{latexsym}
\usepackage{graphicx}
\usepackage{tikz}
\usepackage{bm}

\allowdisplaybreaks[1]

\title[Norms of Schneider--Teitelbaum Polynomials]{Norms of Schneider--Teitelbaum Polynomials \\in $p$-adic Fourier Theory}
\date{}
\author{Yu Katagiri}

\newtheorem{thm}{Theorem}[section]
\newtheorem*{thm*}{Theorem}
\newtheorem{lem}[thm]{Lemma}
\newtheorem*{lem*}{Lemma}

\newtheorem*{prop*}{Proposition}
\newtheorem{cor}[thm]{Corollary}
\newtheorem{cor*}{Corollary}

\theoremstyle{definition}
\newtheorem{dfn}[thm]{Definition}
\newtheorem*{dfn*}{Definition}
\newtheorem{ex}[thm]{Example}
\newtheorem*{ex*}{Example}
\newtheorem{rmk}[thm]{Remark}
\newtheorem*{rmk*}{Remark}

\newtheorem*{conj*}{Conjecture}

\newcommand{\Q}{\mathbb{Q}}
\newcommand{\Z}{\mathbb{Z}}
\newcommand{\C}{\mathbb{C}}
\newcommand{\F}{\mathbb{F}}

\makeatletter
\@addtoreset{equation}{section}

\makeatother

\makeatletter\@namedef{subjclassname@2020}{\textup{2020} Mathematics Subject Classification}\makeatother


\begin{document}
\maketitle

\begin{abstract}
Let $F$ be the unramified quadratic extension of $\mathbb{Q}_p$, and let $P_n(X)$ denote the Lubin--Tate analogues of the binomial polynomials arising in the $p$-adic Fourier theory of Schneider and Teitelbaum. 
For a particular choice of Lubin--Tate logarithm, we determine explicitly the norms of the functions $P_n(x\Omega)$ on the Banach spaces of locally $F$-analytic functions of fixed order, where $\Omega$ is a Lubin--Tate period. 
Our computation is based on a recent valuation formula of Ardakov and Berger for the values $P_n(\Omega)$.
As an application, we improve the Bannai--Kobayashi normalization of these polynomials. 
Each element of the resulting basis has norm one. 
%We also investigate the failure of the Schneider--Teitelbaum polynomials to be orthogonal by means of $t$-orthogonality.
\end{abstract}

\section{Introduction}

Let $p$ be a prime and let $\C_p$ denote the completion of an algebraic closure of $\Q_p$ with the absolute value $|\cdot |$ normalized by $|p|=p^{-1}$.
We regard every finite extension of $\Q_p$ as a subfield of $\C_p$.

The classical theorem of Amice \cite{Am64} gives an explicit description of locally analytic functions on $\Z_p$ in terms of the binomial polynomials $\binom{x}{n}$.
More precisely, for each positive integer $N$, the functions $\lfloor n/p^N \rfloor ! \binom{x}{n}$ form an orthonormal basis of the space $LA_N(\Z_p, \C_p)$ of locally $\Q_p$-analytic functions on $\Z_p$ of order $N$.
This means that every $f \in LA_N(\Z_p, \C_p)$ admits a unique expansion
\begin{align*}
f(x)=\sum_{n=0}^{\infty} a_n \left\lfloor \frac{n}{p^N} \right\rfloor ! \binom{x}{n}
\end{align*}
with $a_n \to 0$ as $n \to \infty$, and its norm is given by $||f||_N=\max_{n \geq 0} \{|a_n|\}$.
In particular, $||f||_N \leq 1$ if and only if $a_n \in \mathcal{O}_{\C_p}$ for all $n \geq 0$, and hence this basis also forms a topological integral basis of the unit ball of $LA_N(\Z_p, \C_p)$.

Schneider and Teitelbaum \cite{ST01} generalized this theory to an arbitrary finite extension of $\Q_p$ by means of Lubin--Tate formal groups.
Their theory, usually referred to as $p$-adic Fourier theory, provides analogues of the binomial polynomials and the Amice transform.
Subsequently, Bannai and Kobayashi \cite{BK16} constructed bases of locally analytic functions and established uniform upper and lower bounds for the norms of their basis elements.
In this note, we determine the exact norms of the Schneider--Teitelbaum polynomials in the case of the unramified quadratic extension of $\Q_p$, for a particular choice of Lubin--Tate logarithm, using a recent valuation formula of Ardakov and Berger \cite{AB24b}.

We briefly recall the objects that will be relevant to the present note.
Let $K$ be a finite extension of $\Q_p$, $q$ be the cardinality of the residue field, and $e$ be the ramification index of $K$.
We fix a uniformizer $\pi$ of $K$ and a Lubin--Tate formal group $\mathcal{G}$ attached to $\pi$.
Let $\log_{\rm LT}(Z)$ denote its formal logarithm.
Choose a generator of the rank-one $\mathcal{O}_K$-module $\operatorname{Hom}_{\mathcal{O}_{\C_p}}(\mathcal{G}, \widehat{\mathbb{G}_m})$ and write it as $\exp (\Omega \log_{\rm LT}(Z))-1$.
The resulting Lubin–Tate period $\Omega \in \mathcal{O}_{\C_p}$ is determined up to multiplication by an element of $\mathcal{O}_K^{\times}$ and satisfies 
\begin{align*}
v_p(\Omega)=\frac{1}{p-1}-\frac{1}{e(q-1)}.
\end{align*}
We define the polynomials $P_n(X) \in K[X]$ by the generating series
\begin{align*}
\exp(X\log_{\rm LT}(Z))=\sum_{n=0}^{\infty} P_n(X)Z^n.
\end{align*}
In this note, we call $P_n(X)$ the $n$-th {\it Schneider--Teitelbaum polynomial}.
These polynomials naturally appear in the $p$-adic Fourier transform.
Indeed, for a locally analytic distribution $\mu$ on $\mathcal{O}_K$, its Fourier transform is given by
\begin{align*}
\mathcal{F}(\mu)(Z)=\sum_{n=0}^{\infty} \mu(P_n(x\Omega))Z^n.
\end{align*}
(See \cite[Theorem 4.6]{BK16} or \cite[Theorems 2.3 and 3.6, and Lemma 4.6.9]{ST01} for the details.)

These polynomials may be regarded as Lubin--Tate analogues of the classical binomial polynomials.
Indeed, in the case $K=\Q_p$, $\mathcal{G}=\widehat{\mathbb{G}_m}$, $\pi=p$ and $\log_{\rm LT}(Z)=\log (1+Z)$, we have $P_n(X)=\binom{X}{n}$.
Thus the Schneider--Teitelbaum theory recovers the classical theory of Amice in this case.
An important integrality property is that $P_n(x \Omega) \in \mathcal{O}_{\C_p}$ if $x \in \mathcal{O}_K$.
This generalizes the familiar fact that the binomial polynomials $\binom{x}{n}$ take values in $\Z_p$ for $x \in \Z_p$, although their coefficients need not be $p$-adically integral.
However, such pointwise integrality does not determine the locally analytic norms considered below, since these norms are defined in terms of the coefficients of local power series expansions.
%Schneider and Teitelbaum studied several properties of $P_n$ in \cite{ST01}.
%We emphasize that $P_n(x \Omega) \in \mathcal{O}_{\C_p}$ if $x \in \mathcal{O}_K$.

We next recall the Banach spaces on which these polynomials will be considered.
Let $N$ be a positive integer.
For $a \in \mathcal{O}_K$, let $A(a+\pi^N\mathcal{O}_K, \C_p)$ denote the space of analytic functions on $a+\pi^N\mathcal{O}_K$.
If $f \in A(a+\pi^N\mathcal{O}_K, \C_p)$ has Taylor expansion $f(a+\pi^Nx)=\sum_{n=0}^{\infty} a_nx^n$, where $a_n \to 0$ as $n \to \infty$, the norm of $f$ is given by
\begin{align*}
||f||_{a, N} \coloneqq \max_{n \geq 0} \{ |a_n| \}.
\end{align*}
We define the space $LA_N(\mathcal{O}_K, \C_p)$ of locally $K$-analytic functions on $\mathcal{O}_K$ of order $N$ by
\begin{align*}
LA_N(\mathcal{O}_K, \C_p) \coloneqq \{ f: \mathcal{O}_K \rightarrow \C_p \mid f|_{a+\pi^N\mathcal{O}_K} \in A(a+\pi^N\mathcal{O}_K, \C_p) \ \text{for any $a \in \mathcal{O}_K$} \},
\end{align*}
equipped with the norm
\begin{align*}
||f||_N \coloneqq \max_{a \in \mathcal{O}_K} \{ ||f||_{a, N} \}.
\end{align*}

Schneider and Teitelbaum showed that $\{ P_n(x\Omega) \}_{n \geq 0}$ forms a topological generating system of $LA_N(\mathcal{O}_K, \C_p)$.
Bannai and Kobayashi \cite{BK16} subsequently obtained quantitative estimates for these polynomials and their expansion coefficients at each fixed analytic level.
More precisely, they showed that the functions
\begin{align*}
e_{N, n}(x) \coloneqq \underline{\gamma} \left( \left\lfloor \frac{n}{q^N} \right\rfloor \right) P_n(x\Omega),
\end{align*}
where $\underline{\gamma}(k) \in \C_p$ satisfies $|\underline{\gamma}(k)|=\underline{\rho}(k) \coloneqq \min_{0 \leq i \leq k} \{|i!/\Omega^i|\}$, form a Banach basis of $LA_N(\mathcal{O}_K, \C_p)$ satisfying
\begin{align*}
c \left| \frac{q}{\pi} \right|^N \leq ||e_{N, n}(x)||_N \leq 1
\end{align*}
for some constant $c$ depending only on $K$.
(See \cite[Proposition 4.5]{BK16} for the details.)
These uniform bounds imply that a series $\sum_{n \geq 0} a_n e_{N, n}$ converges in $LA_N(\mathcal{O}_K, \C_p)$ if and only if $a_n \to 0$.
Thus, the normalization of Bannai and Kobayashi allows the convergence of expansions to be characterized by the decay of their coefficients.
When $K=\Q_p$, choosing $\Omega=1$ and $\underline{\gamma}(k)=k!$ recovers the classical normalized binomial basis appearing in Amice's theory.

Moreover, the $\mathcal{O}_{\C_p}$-submodule topologically generated by $\{ e_{N, n} \}_{n \geq 0}$ is contained in the unit ball and contains a fixed nonzero scalar multiple of the unit ball, although it need not coincide with the entire unit ball.

\if0
Schneider and Teitelbaum showed that $\{ P_n(x\Omega) \}_{n \geq 0}$ forms a topological generating system of $LA_N(\mathcal{O}_K, \C_p)$.
Bannai and Kobayashi \cite{BK16} subsequently gave an explicit construction of the $p$-adic Fourier transform and obtained quantitative estimates for the Schneider--Teitelbaum polynomials.
As an application, they constructed a Banach basis of $LA_N(\mathcal{O}_K, \C_p)$.
More precisely, if we put
%
\begin{align*}
e_{N, n}(x) \coloneqq \underline{\gamma} \left( \left\lfloor \frac{n}{q^N} \right\rfloor \right) P_n(x\Omega) 
\end{align*}
%
with $\underline{\gamma}(k) \in \C_p$ of norm $|\underline{\gamma}(k)|=\underline{\rho}(k) \coloneqq \min_{0 \leq i \leq k} \{|i!/\Omega^i|\}$, the set $\{ e_{N, n}(x)\}_{n \geq 0}$ forms a basis of $LA_N(\mathcal{O}_K, \C_p)$ satisfying
%
\begin{align*}
c \left| \frac{q}{\pi} \right|^N \leq ||e_{N, n}(x)||_N \leq 1
\end{align*}
%
with an explicit constant $c$. (See \cite[Section 4]{BK16} for the details.)
When $K=\Q_p$, this basis reduces to the classical normalized binomial basis appearing in Amice's theory.
%We note that $\{ e_{N, n}(x)\}_{n \geq 0}$ coincides with $\left\{ \lfloor n/p^N \rfloor ! \binom{x}{n} \right\}_{n \geq 0}$ and hence this recovers the Amice's result.
\fi

Motivated by the above result of Bannai and Kobayashi, and by a recent paper by Ardakov and Berger \cite{AB24b}, we compute the norm $|| P_n (x\Omega)||_N$ exactly in the case where $K$ is the unramified quadratic extension of $\Q_p$.
For the rest of the note, let $F$ denote this extension and take a uniformizer $\pi=p$ and the Lubin-Tate formal group whose logarithm is
\begin{align*}
\log_{\rm LT} (Z)=\sum_{i=0}^{\infty} \frac{Z^{q^i}}{p^i},
\end{align*}
where $q=p^2$.
In this setting, Ardakov and Berger proved the following theorem.

\begin{thm}[{\cite[Theorem A]{AB24b}}]\label{thmAB}
For any nonnegative integer $n$, we have
\begin{align*}
v_p(P_n(\Omega))=w(n).
\end{align*}
Here, $w : \Z_{\geq 0} \rightarrow \Q$ is given by
\begin{align}\label{monna}
w(n) \coloneqq \frac{p}{q-1} \sum_{i=0}^h n_ip^{-i}
\end{align}
if $n$ has the $p$-adic expansion $n=\sum_{i=0}^h n_ip^{i}$.
\end{thm}

We emphasize that, for the fixed Lubin--Tate logarithm, the valuation formula in Theorem \ref{thmAB} holds for any choice of the generator defining the period $\Omega$.

Using Theorem \ref{thmAB} and the properties of $w$, we prove our main result.
%we prove the following main result in the present note.

\begin{thm}\label{main1}
Let $N$ be a positive integer.
We write $n=lq^{N+1}+r$ with $l \geq 0$ and $0 \leq r \leq q^{N+1}-1$.
Then we have
\begin{align*}
|| P_n (x\Omega)||_N =p^{-G_N(n)},
\end{align*}
where
\begin{align*}
G_N(n) \coloneqq w(r)+l(u-1)-v_p(l!)
\end{align*}
and $u \coloneqq p/(q-1)$.
\end{thm}

\begin{ex}
If $n<q^{N+1}$, then $G_N(n)=w(n)$ and hence
\begin{align*}
||P_n(x\Omega)||_N=p^{-w(n)}=|P_n(\Omega)|.
\end{align*}
On the other hand, if $n=q^{N+1}$, then
\begin{align*}
G_N(q^{N+1})=u-1<0<uq^{-N-1}=w(q^{N+1})
\end{align*}
and hence
\begin{align*}
||P_{q^{N+1}}(x\Omega)||_N=p^{-u+1}>|P_{q^{N+1}}(\Omega)|.
\end{align*}
\end{ex}

As a consequence, we obtain an explicit norm-one rescaling of the Bannai--Kobayashi basis.
%Our main result allows us to normalize the Bannai--Kobayashi integral basis as follows.

\begin{cor}\label{main2}
For  $n=lq^{N+1}+r$ with $l \geq 0$ and $0 \leq r \leq q^{N+1}-1$, we define
\begin{align*}
\widetilde{e_{N, n}}(x) \coloneqq P_r(\Omega)^{-1} e_{N, n}(x) = P_r(\Omega)^{-1} \underline{\gamma} \left( \left\lfloor \frac{n}{q^N} \right\rfloor \right) P_n(x\Omega).
\end{align*}
Then, we have $|| \widetilde{e_{N, n}}(x) ||_N=1$. 
%In particular, $\widetilde{e_{N, n}}(x) \in LA_N(\OK, \C_p)_0$.
\end{cor}

\begin{rmk}
The normalization in Corollary \ref{main2} gives a basis whose elements all have norm one, as in the classical case of Amice.
However, unlike the Amice basis, this basis is not orthonormal.
For odd $p$, this follows from \cite[Corollary V.2.5]{Sc03}.
For $p=2$, we have
\begin{align*}
P_4(\Omega)^{-1}P_4(x\Omega)-P_1(\Omega)^{-1}P_1(x\Omega)=\frac{\Omega^4}{24P_4(\Omega)}(x^4-x),
\end{align*}
where the scalar on the right-hand side has valuation $-1/2$ by Theorem \ref{thmAB}.
Since $\mathcal{O}_F/2\mathcal{O}_F \simeq \F_4$, the coefficients of the local Taylor expansions of $x^4-x$ at level $N \geq 1$ are divisible by $2$.
Thus, 
\begin{align*}
||P_4(\Omega)^{-1}P_4(x\Omega)-P_1(\Omega)^{-1}P_1(x\Omega)||_N \leq 2^{-1/2}<1,
\end{align*}
which contradicts orthonormality, since the two functions are unit multiples of distinct elements of the normalized basis.

Since a basis is orthonormal if and only if it is a topological integral basis of the unit ball, it follows that $\{ \widetilde{e_{N, n}}(x) \}_{n \geq 0}$ is not a topological integral basis of the $\mathcal{O}_{\C_p}$-module
\begin{align*}
LA_N(\mathcal{O}_F, \C_p)_0 \coloneqq \{ f \in LA_N(\mathcal{O}_F, \C_p) \mid ||f||_N \leq 1\}.
\end{align*}
\end{rmk}

The paper is organized as follows. 
In Section 2, we establish some basic properties of the functions $w$ and $G_N$, including the subadditivity of $G_N$.
In Section 3, we prove Theorem \ref{main1} by estimating the coefficients of the local Taylor expansions of $P_n(x\Omega)$.
Using the subadditivity of $G_N$, we obtain an upper bound for the norm and then show that this bound is attained by identifying a coefficient with the required valuation.

\vspace{10pt}
\noindent
\textsc{Notation:}~ 
Let $F$ be the unramified quadratic extension of $\Q_p$.
For a nonnegative integer $a$, we write the $p$-adic expansion of $a$ by
\begin{align*}
a=\sum_{j=0}^{\infty} a_jp^j
\end{align*}

\vspace{10pt}
\noindent
\textsc{Acknowledgments:}~ The author is grateful to Takao Yamazaki and Shinichi Kobayashi for many helpful comments.
The author would like to thank the anonymous referee for a careful reading of the manuscript and for valuable suggestions that simplified the proof of Theorem \ref{main1}.

\section{Preliminaries}

\subsection{Pointwise valuations}

We first recall some elementary properties of the function $w$ defined in \eqref{monna}, and then use Theorem \ref{thmAB} to establish a lower bound for the valuations of $P_n(x\Omega)$ as $x$ varies over $\mathcal{O}_F$.
%and use them to extend the valuation formula in Theorem \ref{thmAB} from $\Omega$ to $x\Omega$ for $x\in\mathcal{O}_F$.

\begin{lem}[{\cite[Proposition 2.1]{AB24b}}]\label{lemAB}
The function $w$ has the following properties.
\begin{enumerate}
\item For any $n \geq 0$, we have $w(pn)=p^{-1}w(n)$.
\item For any $s \geq 1$ and $0 \leq i \leq p^s-1$, we have $w(p^sn+i)=w(p^sn)+w(i)$.
\item For any $a, b \geq 0$, we have $w(a+b) \leq w(a)+w(b)$.
\end{enumerate}
\end{lem}

\begin{lem}\label{comm_to_AB}
Let $n \geq 0$.
For any $x \in \mathcal{O}_F$, we have
\begin{align*}
|P_n(x \Omega)| \leq p^{-w(n)}.
\end{align*}
Moreover, equality holds when $x \in \mathcal{O}_F^{\times}$.
%Moreover, if $x \in \mathcal{O}_F^{\times}$, the equality holds.
\end{lem}

\begin{proof}
First, suppose that $x \in \mathcal{O}_F^{\times}$.
Since $x\Omega$ is also a Lubin--Tate period arising from a generator of $\operatorname{Hom}_{\mathcal{O}_{\C_p}}(\mathcal{G}, \widehat{\mathbb{G}_m})$, Theorem \ref{thmAB} gives $v_p(P_n(x\Omega))=w(n)$.

Suppose that $x \in p\mathcal{O}_F$.
Since both $1$ and $x-1$ belong to $\mathcal{O}_F^{\times}$, \cite[Lemma 4.2.4]{ST01} gives
\begin{align*}
P_n(x\Omega)=\sum_{i=0}^n P_i(\Omega)P_{n-i}((x-1)\Omega).
\end{align*}
By the first part of the proof and Lemma \ref{lemAB}(3), we have
\begin{align*}
|P_n(x\Omega)| &\leq \max_{0 \leq i \leq n} \left\{ |P_i(\Omega)| |P_{n-i}((x-1)\Omega)| \right\} \\
&=\max_{0 \leq i \leq n} \left\{ p^{-w(i)-w(n-i)} \right\} \leq p^{-w(n)}.
\end{align*}
This completes the proof.
\end{proof}

\subsection{Inequalities for $w$ and $G_N$}

In this subsection, we compare the functions $w$ and $G_N$ and establish the subadditivity of $G_N$.
These properties will play a key role in the proof of Theorem \ref{main1}.

If we put $h(l) \coloneqq -l(1-u)-v_p(l!)$ for $l \geq 0$, the function $G_N$ in Theorem \ref{main1} can be written as
\begin{align*}
G_N(n)=w(r)+h(l)
\end{align*}
for $n=lq^{N+1}+r$ with $l \geq 0$ and $0 \leq r \leq q^{N+1}-1$.
We note that $h(l) \leq 0$ and that 
\begin{align}\label{subadd_h}
h(l_1+l_2) \leq h(l_1)+h(l_2) 
\end{align}
for any $l_1, l_2 \geq 0$, since
\begin{align*}
v_p((l_1+l_2)!) \geq v_p(l_1!)+v_p(l_2!).
\end{align*}

\begin{lem}\label{ine_w_G}
\begin{enumerate}
\item For any $n, s \geq 0$, we have
\begin{align*}
w(p^sn) \leq u \sum_{i=0}^\infty n_i.
\end{align*}
\item For any $n \geq 0$, we have $w(n) \geq G_N(n)$. Moreover, equality holds if and only if $0 \leq n \leq q^{N+1}-1$.
\end{enumerate}
\end{lem}

\begin{proof}
\begin{enumerate}
\item Since $p^sn$ has the $p$-adic expansion $\sum_{i=0}^{\infty} n_ip^{i+s}$, we have
\begin{align*}
w(p^sn)=u\sum_{i=0}^{\infty} n_ip^{-i-s} \leq u \sum_{i=0}^{\infty} n_i.
\end{align*}
\item If we write $n=lq^{N+1}+r$ with $l \geq 0$ and $0 \leq r \leq q^{N+1}-1$, we see that
\begin{align*}
w(n)-G_N(n)&=w(lq^{N+1})+w(r)-(w(r)+h(l)) \\
&=w(lq^{N+1})-h(l) \geq 0.
\end{align*}
Here, we used Lemma \ref{lemAB}(2) in the first equality.
Moreover, $w(n) = G_N(n)$ if and only if $w(lq^{N+1})=h(l)=0$, hence $l=0$.
\end{enumerate}
\end{proof}

\begin{lem}\label{subadd_G}
For any $n_1, n_2 \geq 0$, we have
\begin{align*}
G_N(n_1+n_2) \leq G_N(n_1)+G_N(n_2).
\end{align*}
\end{lem}

\begin{proof}
For $i=1, 2$, we write $n_i=l_iq^{N+1}+r_i$ with $l_i \geq 0$ and $0 \leq r_i \leq q^{N+1}-1$.
If $r_1+r_2 \leq q^{N+1}-1$, the claim follows from Lemma \ref{lemAB}(3) and \eqref{subadd_h}.

We suppose that $r_1+r_2 \geq q^{N+1}$ and write $r_1+r_2 =q^{N+1}+r'$ with $r' \geq 0$.
Then we have
\begin{align*}
w(r_1)+w(r_2)& \geq w(r_1+r_2)=w(q^{N+1}+r') \\
&=w(q^{N+1})+w(r') \geq w(r')
\end{align*}
by Lemma \ref{lemAB}(2) and (3).
Hence, since $n_1+n_2=(l_1+l_2+1)q^{N+1}+r'$, we obtain
\begin{align*}
G_N(n_1+n_2)&=w(r')+h(l_1+l_2+1) \\
&\leq w(r_1)+w(r_2)+h(l_1)+h(l_2)+h(1) \\
&<G_N(n_1)+G_N(n_2)
\end{align*}
by \eqref{subadd_h}.
\end{proof}

\section{Proof of Theorem \ref{main1}}

\subsection{Local expansions and coefficient estimates}
We next derive an expansion that will be used to compute the norm $||P_n(x\Omega)||_N$.
By the definition of $P_n$, we have
\begin{align*}
\sum_{n=0}^{\infty} P_n(X)Z^n=\exp (X \log_{\rm LT}(Z))=\prod_{j=0}^{\infty} \exp( p^{-j}XZ^{q^j}).
\end{align*}
For $a \in \mathcal{O}_F$, we substitute $X=(a+p^NY)\Omega$ into this identity.
This gives
%To compute $||P_n(x\Omega)||_N$, substituting $X=(a+p^NY)\Omega$ with $a \in \mathcal{O}_F$, we obtain 
%
\begin{align*}
\sum_{n=0}^{\infty} P_n((a+p^NY)\Omega)Z^n&=\exp (a\Omega \log_{\rm LT}(Z)) \exp (p^N\Omega Y\log_{\rm LT}(Z)) \\
&=\left( \sum_{n=0}^{\infty} P_n(a\Omega)Z^n \right) \cdot \prod_{j=0}^{\infty} \exp( p^{N-j}\Omega YZ^{q^j}) \\
&=\left( \sum_{n=0}^{\infty} P_n(a\Omega)Z^n \right) \cdot \prod_{j=0}^{\infty} \sum_{\nu_j \geq 0} \frac{( p^{N-j}\Omega Y)^{\nu_j}}{\nu_j !}Z^{q^j\nu_j}.
\end{align*}
For a sequence $\nu=(\nu_j)_{j \geq 0}$ of nonnegative integers with finite support, we use the following notation:
\begin{align*}
M(\nu) \coloneqq \sum_{j=0}^{\infty} \nu_jq^j, \qquad |\nu| \coloneqq \sum_{j=0}^\infty \nu_j.
\end{align*}
Comparing the coefficients of $Z^n$ in the above identity, we obtain
\begin{align}\label{P_n on discs}
P_n((a+p^NY)\Omega)=\sum_{m=0}^n \sum_{\substack{\nu \\ M(\nu)=m}} P_{n-m}(a\Omega) \frac{p^{\sum_{j \geq 0}(N-j)\nu_j} \Omega^{|\nu|}}{\prod_{j \geq 0}\nu_j!} Y^{|\nu|}.
\end{align}
To estimate the $p$-adic valuations of the terms appearing in the right-hand side of \eqref{P_n on discs}, we introduce the following auxiliary function.

\begin{dfn}\label{principal lem}
For nonnegative integers $j$ and $t$, we define
\begin{align*}
\phi_{N, j}(t) \coloneqq (u+N-j)t-v_p(t!).
\end{align*}
\end{dfn}

For $0 \leq m \leq n$ and $\nu=(\nu_j)_{j \geq 0}$ with $M(\nu)=m$, we see that 
\begin{align}\label{val of target}
\begin{split}
&v_p \left( P_{n-m}(a\Omega) \frac{p^{\sum_{j \geq 0}(N-j)\nu_j} \Omega^{|\nu|}}{\prod_{j \geq 0}\nu_j!} \right) \\
=& v_p(P_{n-m}(a\Omega)) + \sum_{j \geq 0} v_p \left( \frac{p^{(N-j)\nu_j} \Omega^{\nu_j}}{\nu_j!} \right) \\
\geq & w(n-m) + \sum_{j \geq 0} \phi_{N, j}(\nu_j) \eqqcolon C(\nu),
\end{split}
\end{align}
where the last inequality follows from Lemma \ref{comm_to_AB}.

\begin{lem}\label{phi_G}
For any $t, j \geq 0$, we have $\phi_{N, j}(t) \geq G_N(tq^j)$.
Moreover, equality holds if and only if either of the following conditions holds:
\begin{enumerate}
\item $t=0$,
\item $j=N$ and $q \mid t$,
\item $j=N+1$.
\end{enumerate}
\end{lem}

\begin{proof}
We consider the three cases: $0 \leq j \leq N-1$, $j=N$, and $j \geq N+1$.

If $0 \leq j \leq N-1$, Lemma \ref{ine_w_G} implies that
\begin{align*}
&\phi_{N, j}(t) - G_N(tq^j) \\
\geq &(u+N-j)t-v_p(t!)-w(tq^j) \\
\geq &(u+N-j)t-\frac{1}{p-1}t+\left( \frac{1}{p-1}-u \right)\sum_{i=0}^\infty t_i \\
\geq & \left(N-j-\frac{1}{q-1} \right)t+\frac{1}{q-1}\sum_{i=0}^\infty t_i \geq 0.
\end{align*}
Equality in the last inequality holds if and only if $t=0$, and if $t=0$, then $\phi_{N, j}(0) = G_N(0)=0$.

We suppose that $j=N$ and write $t=aq+b$ with $a \geq 0$ and $0 \leq b \leq q-1$.
We note that 
\begin{align}\label{obs1}
v_p(t!)=(p+1)a+v_p(a!)+v_p(b!)
\end{align}
and that
\begin{align}\label{obs2}
ub-v_p(b!)=w(b)
\end{align}
by direct computation.
Since $tq^N=aq^{N+1}+bq^N$ and $0 \leq bq^N <q^{N+1}$, \eqref{obs1} implies that
\begin{align*}
&\phi_{N, N}(t) - G_N(tq^N) \\
=&u(aq+b)-v_p(t!)-\{ w(bq^N)-a(1-u)-v_p(a!) \} \\
=&ub-v_p(b!)-w(bq^N)+a\{ uq-(p+1)+(1-u) \}.
\end{align*}
Since $uq-(p+1)+(1-u)=u(q-1)-p=0$, we obtain
\begin{align*}
\phi_{N, N}(t) - G_N(tq^N)=w(b)-w(bq^N)=(1-p^{-2N})w(b) \geq 0
\end{align*}
by \eqref{obs2} and Lemma \ref{lemAB}(1).
We also see that $\phi_{N, N}(t) = G_N(tq^N)$ if and only if $b=0$, hence $q$ divides $t$.

We finally suppose that $j \geq N+1$.
If we write $j=N+1+s$ with $s \geq 0$, then
\begin{align*}
\phi_{N, N+1+s}(t)&=(u-s-1)t-v_p(t!) \\
&=(u-s-1)t-\frac{1}{p-1}t+\frac{1}{p-1}\sum_{i=0}^\infty t_i
\end{align*}
and
\begin{align*}
G_N(tq^j)&=h(tq^s)=-tq^s(1-u)-v_p((tq^s)!) \\
&=-tq^s(1-u)-\frac{1}{p-1}tq^s+\frac{1}{p-1}\sum_{i=0}^\infty t_i.
\end{align*}
Hence we obtain
\begin{align*}
&\phi_{N, N+1+s}(t)-G_N(tq^j) \\
=&t \left\{ \frac{q(q^s-1)}{q-1}-s \right\}=t\sum_{i=1}^s (q^i-1) \geq 0.
\end{align*}
We also find that equality holds if and only if $t=0$ or $s=0$.
\end{proof}

\begin{cor}\label{thm_zenhan}
We have
\begin{align*}
|| P_n (x\Omega)||_N \leq p^{-G_N(n)}.
\end{align*}
\end{cor}

\begin{proof}
For $a \in \mathcal{O}_F$, we write
\begin{align*}
P_n((a+p^NY)\Omega)=\sum_{d=0}^n A_{a, d} Y^d,
\end{align*}
where
\begin{align*}
A_{a, d}=\sum_{\substack{\nu \\ |\nu|=d, M(\nu)\leq n}} P_{n-M(\nu)}(a\Omega) \frac{p^{\sum_{j \geq 0}(N-j)\nu_j} \Omega^{d}}{\prod_{j \geq 0}\nu_j!}.
\end{align*}
For $\nu=(\nu_j)_{j \geq 0}$ with $|\nu|=d$ and $m \coloneqq M(\nu) \leq n$, we have 
\begin{align*}
&v_p \left( P_{n-m}(a\Omega) \frac{p^{\sum_{j \geq 0}(N-j)\nu_j} \Omega^{|\nu|}}{\prod_{j \geq 0}\nu_j!} \right) \\
\geq & C(\nu) \geq G_N(n-m)+\sum_{j=0}^\infty G_N(\nu_jq^j) \\
\geq & G_N(n-m+\sum_{j=0}^\infty \nu_jq^j)=G_N(n)
\end{align*}
by Lemma \ref{phi_G} and Lemma \ref{subadd_G}.
This proves the assertion.
\end{proof}

\subsection{The reverse inequality}

To complete the proof of Theorem \ref{main1}, we consider the norm $||P_n(x\Omega)||_{1, N}$ on the residue disk $1+p^N\mathcal{O}_F$.
If we write $n=lq^{N+1}+r$ with $l \geq 0$ and $0 \leq r \leq q^{N+1}-1$, then, by \eqref{P_n on discs}, the coefficient of $Y^l$ in $P_n((1+p^NY)\Omega)$ is given by
\begin{align}\label{target_kohan}
 \sum_{\substack{\nu \\ |\nu|=l, M(\nu) \leq n}} P_{n-M(\nu)}(\Omega) \frac{p^{\sum_{j \geq 0}(N-j)\nu_j} \Omega^{l}}{\prod_{j \geq 0}\nu_j!}.
\end{align}
%
%We compute the $p$-adic valuation of \eqref{target_kohan}.

To compute the $p$-adic valuation of \eqref{target_kohan}, we define the sequence $\nu'=(\nu'_j)_{j \geq 0}$ by
\begin{align*}
\nu'_j=
\begin{cases}
l & \text{if $j=N+1$}, \\
0 & \text{if $j\neq N+1$}.
\end{cases}
\end{align*}

\begin{lem}\label{kantan}
We have
\begin{align*}
v_p \left( P_{n-M(\nu')}(\Omega) \frac{p^{\sum_{j \geq 0}(N-j)\nu'_j} \Omega^{l}}{\prod_{j \geq 0}\nu'_j!} \right)=G_N(n).
\end{align*}
\end{lem}

\begin{proof}
The left-hand side equals
\begin{align*}
v_p \left( P_{n-lq^{N+1}}(\Omega) \frac{p^{-l} \Omega^{l}}{l!} \right)=v_p(P_r(\Omega))+l(u-1)-v_p(l!)=G_N(n).
\end{align*}
\end{proof}

\begin{lem}\label{taihen}
Let $\nu=(\nu_j)_{j\geq 0}$ be a sequence of nonnegative integers with finite support satisfying $|\nu|=l$ and $M(\nu) \leq n$.
If $\nu \neq \nu'$, then we have
\begin{align*}
v_p \left( P_{n-M(\nu)}(\Omega) \frac{p^{\sum_{j \geq 0}(N-j)\nu_j} \Omega^{l}}{\prod_{j \geq 0}\nu_j!} \right)>G_N(n).
\end{align*}
\end{lem}

\begin{proof}
Let $m \coloneqq M(\nu)$.
It suffices to prove that $C(\nu)>G_N(n)$.
As we had seen in the proof of Corollary \ref{thm_zenhan}, since
\begin{align}\label{at_least}
C(\nu) \geq G_N(n-m)+\sum_{j=0}^\infty G_N(\nu_jq^j) \geq G_N(n),
\end{align}
it is enough to check that the first inequality in \eqref{at_least} is strict.

If $\phi_{N, j}(\nu_j) =G_N(\nu_jq^j)$ for all $j \geq 0$, Lemma \ref{phi_G} implies that $\nu=(\nu_j)_{j \geq 0}$ has the form
\begin{align}\label{N_to_N+1}
\nu_j=
\begin{cases}
0 & \text{if $j \neq N, N+1$}, \\
qa & \text{if $j=N$}, \\
\nu_{N+1} & \text{if $j=N+1$}.
\end{cases}
\end{align}
If $\nu_j$'s are given by \eqref{N_to_N+1}, then $\nu_{N+1}=l-qa$ since $|\nu|=l$.
We also see that
\begin{align*}
n-m=(lq^{N+1}+r)-\{ q^{N+1}a+q^{N+1}(l-qa)\}=(q-1)aq^{N+1}+r.
\end{align*}
Hence, Lemma \ref{ine_w_G}(2) implies that $w(n-m)>G_N(n-m)$ for $a \geq 1$.
Therefore, if $\nu \neq \nu'$, the first inequality in \eqref{at_least} must be strict.
\end{proof}

\begin{proof}[Proof of Theorem \ref{main1}]
Lemma \ref{kantan} and Lemma \ref{taihen} imply that the $p$-adic valuation of \eqref{target_kohan} is given by $G_N(n)$.
Combining this with Corollary \ref{thm_zenhan}, we obtain the theorem.
\end{proof}

\if0
\begin{ex}
If $n<q^{N+1}$, then $G_N(n)=w(n)$ and hence
%
\begin{align*}
||P_n(x\Omega)||_N=p^{-w(n)}=|P_n(\Omega)|.
\end{align*}
%
On the other hand, if $n=q^{N+1}$, then
%
\begin{align*}
G_N(q^{N+1})=u-1<0<uq^{-N-1}=w(q^{N+1})
\end{align*}
%
and hence
%
\begin{align*}
||P_{q^{N+1}}(x\Omega)||_N=p^{-u+1}>|P_{q^{N+1}}(\Omega)|.
\end{align*}
%
\end{ex}
\fi

We conclude by proving Corollary \ref{main2}.

\begin{proof}[Proof of Corollary \ref{main2}]
\cite[Proposition 3.4]{BK16} implies that
\begin{align*}
\left| \underline{\gamma}\left( \left\lfloor \frac{n}{q^N} \right\rfloor \right) \right|=\underline{\rho} \left( \left\lfloor \frac{n}{q^N} \right\rfloor \right)=\left| \frac{(lq)!}{\Omega^{lq}}\right|.
\end{align*}
Since
\begin{align*}
v_p \left(\frac{(lq)!}{\Omega^{lq}}\right)=-l(u-1)+v_p(l!), 
\end{align*}
it follows from Theorem \ref{main1} that $||e_{N, n}(x)||_N=p^{-w(r)}$.
Hence Theorem \ref{thmAB} gives the desired equality.
\end{proof}

\vspace{10pt}

\noindent
Institute of Mathematics for Industry, Kyushu University,\\
744, Motooka, Nishi-ku, Fukuoka, 819-0395, Japan,\\
%Mathematical Institute, Graduate School of Science, Tohoku University,\\
%6-3 Aramakiaza, Aoba, Sendai, Miyagi 980-8578, Japan.\\
E-mail address: \textbf{yu.katagiri.s3@gmail.com}\\


\begin{bibdiv}
\begin{biblist}

\bib{Am64}{article}{
   author={Amice, Yvette},
   title={Interpolation $p$-adique},
   journal={Bull. Soc. Math. France},
   volume={92},
   date={1964},
   pages={117-180},
}

\bib{AB24a}{article}{
   author={Ardakov, Konstantin},
   author={Berger, Laurent},
   title={Bounded functions on the character variety},
   journal={M\"unster J. Math.},
   volume={17},
   date={2024},
   number={1},
   pages={1--83},
 
}

\if0
\bib{AB24b}{article}{
   author={Ardakov, Konstantin},
   author={Berger, Laurent},
   title={$p$-adic Fourier theory for $\mathbb{Q}_{p^2}$ and the Monna map},
   date={2024},
   note={arXiv:2405.05726},
}
\fi
\bib{AB24b}{article}{
   author={Ardakov, Konstantin},
   author={Berger, Laurent},
   title={$p$-adic Fourier theory for $\mathbb{Q}_{p^2}$ and the Monna map},
   journal={Confluentes Math.},
   volume={18},
   date={2026},
   pages={71--79},

}

\bib{BK16}{article}{
   author={Bannai, Kenichi},
   author={Kobayashi, Shinichi},
   title={Integral structures on $p$-adic Fourier theory},
   journal={Ann. Inst. Fourier (Grenoble)},
   volume={66},
   date={2016},
   number={2},
   pages={521--550},
}

\bib{Sc03}{thesis}{
   author={Scanlon, M. G. T.},
   title={$p$-adic Fourier analysis},
   type={Ph.D. thesis},
   school={University of Durham},
   date={2003},
}





\bib{ST01}{article}{
   author={Schneider, Peter},
   author={Teitelbaum, Jeremy},
   title={$p$-adic Fourier theory},
   journal={Doc. Math.},
   volume={6},
   date={2001},
   pages={447-481},
}



\end{biblist}
\end{bibdiv}
\end{document}